\documentclass[aop,MSNbibl,dvips]{arximspdf}

\doi{10.1214/12-AOP824} 
\volume{41}
\issue{5}
\pubyear{2013}
\firstpage{3181}
\lastpage{3200}

\makeatletter
\newcommand{\rrvert}{\vert}
\newtheorem{theorem}{Theorem}[section]

\newtheorem{lemma}[theorem]{Lemma}
\newtheorem{prop}[theorem]{Proposition}
\newtheorem{corollary}[theorem]{Corollary}
\newcommand{\mod}{\ \mathrm{mod}\ }

\newcommand{\eqref}[1]{(\ref{#1})}
\makeatother

\begin{document}
\begin{frontmatter}

\title{Ergodicity of Poisson products and applications}
\runtitle{Ergodicity of Poisson products}

\begin{aug}
\author[A]{\fnms{Tom} \snm{Meyerovitch}\corref{}\ead[label=e1]{mtom@math.bgu.ac.il}}
\runauthor{T. Meyerovitch}
\affiliation{Ben-Gurion University of the Negev}
\address[A]{Department of Mathematics\\
Ben-Gurion University of the Negev\\
Be'er Sheva 84105\\
Israel\\
\printead{e1}} 
\end{aug}

\received{\smonth{7} \syear{2011}}
\revised{\smonth{11} \syear{2012}}

%
\begin{abstract}
In this paper we study the Poisson process over a $\sigma$-finite measure-space
equipped with a measure preserving transformation or a group of measure
preserving transformations. 
For a measure-preserving transformation $T$ acting on a $\sigma$-finite
measure-space $X$, the \emph{Poisson suspension} of $T$ is the
associated probability preserving transformation $T_*$ which acts on
realization of the Poisson process over $X$. We prove ergodicity of the
\emph{Poisson-product} $T \times T_*$ under the assumption that $T$ is
ergodic and conservative.
We then show, assuming ergodicity of $T \times T_*$, that it is
impossible to deterministically perform natural equivariant operations:
\emph{thinning}, \emph{allocation} or \emph{matching}. In contrast,
there are well-known results in the literature demonstrating the
existence of \emph{isometry} equivariant thinning, matching and
allocation of homogenous Poisson processes on $\mathbb{R}^d$.
We also prove ergodicity of the ``first return of left-most
transformation'' associated with a measure preserving transformation on
$\mathbb{R}_+$, and discuss ergodicity of the Poisson-product of
measure preserving group actions, and related spectral properties.

\end{abstract}

%

\end{frontmatter}

\section{Introduction}\label{sec1}
It is straightforward that the distribution of a homogenous Poisson
point process on $\mathbb{R}^d$ is preserved
by isometries. In the literature, various \emph
{translation-equivariant} and \emph{isometry-equivariant}
operations on Poisson process have been considered:
\begin{itemize}
\item \textit{Poisson thinning}: A (deterministic) \emph
{Poisson-thinning} is a rule for selecting a
subset of the points in the Poisson process which are equal in
distribution to
a lower intensity homogenous Poisson process.
Ball~\cite{ballthinning} demonstrated a deterministic
Poisson-thinning on $\mathbb{R}$ which was \emph{translation
equivariant}---that is, if a translation is applied to the original process,
the new points selected are translations of the original ones by the
same vector.
This was extended and refined by Holroyd, Lyons and Soo \cite
{holroydlyonssoopoissonsplitting2011}
to show that for any $d \ge1$,
there is an \emph{isometry-equivariant} Poisson-thinning 
on $\mathbb{R}^d$.
\item \textit{Poisson allocation}:
Given a realization $\omega$
of a Poisson process on $\mathbb{R}^d$, a~\emph{Poisson allocation}
partitions $\mathbb{R}^d$ up to measure $0$
by assigning to each point in $\omega$ a \emph{cell} which is a
finite-measure subset of $\mathbb{R}^d$.
Hoffman, Holroyd and Peres~\cite{hoffmanholroydperesstableallocation}
constructed an isometry-equivariant allocation scheme for any
stationary point process of finite intensity.
The above allocation scheme had the characteristic property of being ``stable.''
Subsequent work demonstrated isometry-equivariant Poisson allocations
with other nice properties
such as connectedness of the allocated cells~\cite{krikunallocation}
or good stochastic bounds on the diameter of the cells \cite
{chatterjeepeledromikgravitationalallocation}.
\item \textit{Poisson matching}: A \emph{Poisson matching} is a
deterministic scheme which finds a perfect
matching of two identically distributed independent Poisson processes.
Different isometry-equivariant Poisson matching schemes have been constructed
\cite{holrodydmatching2011,holroydpemantleperesschrammpoissonmatching}.
\end{itemize}

Consider a transformation 
of $\mathbb{R}^d$ which preserves Lebesgue measure. Does there
exist a Poisson thinning which is equivariant with respect to the
given transformation? What about an equivariant Poisson allocation
or matching?

To have a couple of examples in mind, consider the following
transformations $T_{\mathrm{RW}},T_{\mathrm{Boole}}\dvtx \mathbb{R}\to\mathbb{R}$ of the
real line given by
%
\begin{equation}
T_{\mathrm{RW}}(x)= \lfloor x \rfloor+ (2x \mod1) -1 + 2\cdot1_{(0,
{1}/{2}]}(x
\mod1)
\end{equation}
and
%
\begin{equation}
T_{\mathrm{Boole}}(x)= x - \frac{1}{x}
\end{equation}
$T_{\mathrm{Boole}}$ is known as Boole's transformation. It is a
is a classical example of an ergodic transformation preserving Lebesgue measure.
See~\cite{adlerweiss73} for a proof of ergodicity and discussions of
this transformation.
You may notice that $T_{\mathrm{RW}}$ is isomorphic to the shift map on the
space of forward trajectories of the simple random walk
on~$\mathbb{Z}$.


From our perspective, it is natural (although mathematically
equivalent) to consider an abstract standard $\sigma$-finite measure
space $(X,\mathcal{B},\mu)$,
instead of $\mathbb{R}^d$ with Lebesgue measure. 
We consider a Poisson point process on this space,
which denoted by $(X^*,\mathcal{B}^*,\mu^*)$.
Any measure preserving transformation $T\dvtx X \to X$ naturally induces a map
$T_*\dvtx X^* \to X^*$ on the Poisson process.
This transformation $T_*$ is the \emph{Poisson suspension} of $T$ \cite
{roy2009}.


We prove the following theorem:

\begin{theorem}\label{thmnopoissonthin}
Let $T\dvtx X \to X$ be any conservative and ergodic measure preserving
transformation of $(X,\mathcal{B},\mu)$ with $\mu(X)=\infty$. There
does not exist a \mbox{$T$-equivariant} Poisson thinning, allocation or matching.
\end{theorem}

We prove Theorem~\ref{thmnopoissonthin} by studying ergodic
properties of the map
$T \times T_*$, which acts on the product space $(X \times X^*, \mathcal
{B}\times\mathcal{B}^*, \mu^* \times\mu)$.
We refer to this system as the \emph{Poisson-product}\vadjust{\goodbreak} associated with $T$.
The space $X\times X^*$ can be considered as a countable set of
``indistinguishable'' points in $X$, with a unique ``distinguished'' point.
The Poisson-product $T\times T_*$ acts on this by applying the same map
$T$ to each point, including the distinguished point.

Our main result about Poisson-products is the following theorem:

\begin{theorem}
\label{thmpoissonproductergodic} Let $(X,\mathcal{B},\mu,T)$ be a
conservative, measure-preserving transformation with
$\mu(X)=\infty$. Then the Poisson-product $T\times T_*$ is ergodic
if and only if $T$ is ergodic.
\end{theorem}

Before concluding the introduction and proceeding with the details, we
recall a couple of results regarding nonexistence of certain
equivariant operations on Poisson processes. Evans proved in
\cite{evens2010} that with respect to any noncompact group of
\emph{linear} transformations there is no invariant Poisson-thinning on
$\mathbb{R}^d$. Gurel-Gurevich and Peled proved the nonexistence of
translation equivariant \emph{Poisson thickening} on the real line \cite
{gurelpeled}, which means that there is no measurable function on
realizations of the a homogenous Poisson process that sends a Poisson
process to a higher intensity homogenous Poisson process.

This paper is organized as follows: In Section~\ref{secprelim} we
briefly provide some terminology and necessary
background. Section~\ref{secproofpoissonproductergodic}
contains a short proof of Theorem~\ref{thmpoissonproductergodic}
stated above, based on previous work in 
ergodic theory.
In Section~\ref{secthinnings}
we prove any $T$-equivariant thinning is trivial, assuming $T \times
T_*$ is ergodic.
In Section~\ref{secallocationmatching} we show that under the same assumptions
there are no $T$-equivariant Poisson allocations or Poisson matchings,
using an intermediate result about nonexistence of positive
equivariant maps into~$L^1$.
Section~\ref{secFROL} discusses the ``leftmost position
transformation'' and contains a proof of ergodicity,
yet another application of Theorem~\ref{thmpoissonproductergodic}.
Section~\ref{secgroupactions} is a discussion of ergodicity of
Poisson products for measure preserving group actions.



\section{Preliminaries}\label{secprelim}

In this section we briefly recall some definitions and background from
ergodic theory required for the rest of the paper. We also recall some
properties of the Poisson point process on a $\sigma$-finite measure
space.

\subsection{Ergodicity, conservative transformations and induced
transformations}
Throughout this paper
$(X,\mathcal{B},\mu)$ is a standard $\sigma$-finite measure space. We
will mostly be interested in the case where $\mu(X)=\infty$.
Also throughout the paper, $T\dvtx X \to X$ is a measure preserving transformation,
unless explicitly stated otherwise,
where $T$ denotes an action of a group by measure preserving
transformations of $(X,\mathcal{B},\mu)$.
The collection of measurable
sets of positive measure by will be denoted by $\mathcal{B}^+:= \{B
\in\mathcal{B} \dvtx   \mu(B)>0\}$.

Recall that $T$ is \emph{ergodic} if any set $A \in\mathcal{B}$ which
is $T$-invariant has either $\mu(A)=0$ or $\mu(A^c)=0$. Equivalently,\vadjust{\goodbreak}
$T$ is ergodic if any measurable function $f\dvtx X \to\mathbb{R}$
satisfying $f \circ T= f$ $\mu$-almost everywhere is constant on a set
of full measure.

A set $W \in\mathcal{B}$ is called a \emph{wandering set} if $\mu
(T^{-n}W \cap W)=0$ for all $n > 0$.
The transformation $T$ is called \emph{conservative} if there are no
wandering sets in $\mathcal{B}^+$.
The \emph{Poincar\'{e} recurrence theorem} asserts that any $T$ which
preserves a \emph{finite} measure is conservative.

For a conservative $T$ and $A \in\mathcal{B}^+$, the \emph{first
return time function} is
defined for $x \in A$ by $\varphi_A(x)= \min\{ n \ge1 \dvtx   T^{n}(x) \in
A\}$. $\varphi_A$ is finite $\mu$-a.e; this is a direct consequence of
$T$ being conservative.

The \emph{induced transformation} on $A$ is defined by $T_A(x):=
T^{\varphi_A(x)}(x)$.
If $T$ is conservative and ergodic and $A \in\mathcal{B}^+$, $T_A\dvtx A
\to A$ is a conservative,
ergodic transformation of $(A,\mathcal{B} \cap A,\mu\mid_A)$.

See~\cite{aarobook} for a comprehensive introduction to ergodic theory
of infinite measure preserving transformations.

\subsection{Cartesian product transformations}
Suppose $T$ is conservative, and $S\dvtx Y \to Y$ is a probability
preserving transformation
of $(Y,\mathcal{C},\nu)$, namely $\nu(Y)=1$.
It follows (as in Proposition 1.2.4 in~\cite{aarobook})
that the \emph{Cartesian product transformation} $T\times S\dvtx X\times Y
\to X\times Y$ is a conservative,
measure-preserving transforation of the Cartesian product measure-space
$(X\times Y,\mathcal{B} \otimes\mathcal{C},\mu\times\nu)$.

\subsection{\texorpdfstring{$L^\infty$-eigenvalues of measure preserving transformations}
{L infinity-eigenvalues of measure preserving transformations}}

A function $f \in L^{\infty}(X,\mathcal{B},\mu)$ is an
\emph{$L^\infty$-eigenfunction} of $T$ if $f\ne0$ and $Tf=\lambda
f$ for some $\lambda\in\mathbb{C}$. The corresponding $\lambda$ is
called an $L^\infty$-\emph{eigenvalue} of $T$. 
We briefly recall some well-known results:

If $T$ is ergodic and $f$ is an
$L^\infty$-eigenfunction, it follows that $|f|$ is constant
almost-everywhere. The $L^\infty$-eigenvalues of $T$ are
\[
e(T):= \bigl\{ \lambda\in\mathbb{C} \dvtx \exists f\in L^{\infty}(X,
\mathcal {B},\mu) f\ne0 \mbox{ and } Tf=\lambda f \bigr\}.
\]

If $T$ is conservative, then $|\lambda| = 1$ for any eigenvalue $\lambda
$, for otherwise
the set
%
\[
\bigl\{x \in X \dvtx \bigl|f(x) \bigr| \in\bigl(|\lambda|^{k},|\lambda|^{k+1}\bigr]
\bigr\}
\]
would be a
nontrivial wandering set for some $k \in\mathbb{Z}$ if
$|\lambda|>1$. Thus, for any conservative transformation $T$, $e(T)$ is
a subset if the unit sphere
\[
\mathbb{S}^{1}= \bigl\{ x \in\mathbb{C} \dvtx |x|=1\bigr\}.
\]

$e(T)$ is a group with respect to multiplication, and
carries a natural Polish topology, with respect to which the natural
embedding in
$\mathbb{S}^{1}$ is continuous.

When $T$ preserves a finite measure, $e(T)$ is at most countable.
For a general infinite-measure preserving $T$, however, $e(T)$
can be uncountable, and quite ``large,'' for instance, the arbitrary
Hausdorff dimension
$\alpha\in(0,1)$. Importantly for us, however, there are limitations
on how ``large''\vadjust{\goodbreak} $e(T)$ can be. For instance, $e(T)$ is a \emph{weak
Dirichlet} set. This means that
\[
\liminf_{n \to\infty}\int\bigl|1- \chi_n(s)\bigr|\,dp(s)=0
\]
whenever $p$ is a
probability measure on $\mathbb{S}^{1}$ with $p(e(T))=1$, and $\chi_n(s):=\exp (2\pi i n s )$.
In particular the set $e(T)$ has measure zero with respect to Haar
measure on $\mathbb{S}^1$.

We refer the reader to existing literature for further details
\cite
{aarobook,aaronadkarni1987,nadkarnispectraldsbook,schmidtspectra1982}.

\subsection{The $L^2$-spectrum}

Let 
$U_T\dvtx L^2(\mu)\to L^2(\mu)$ denote the unitary operator defined by
$U_T(f):=f \circ T$.


The \emph{spectral type} of a unitary operator $U$ on a Hilbert
space $H$, denoted $\sigma_U$, is a positive measure on $\mathbb{S}^1$
satisfying
\begin{longlist}[(a)]
\item[(a)]
\[
\bigl\langle U^nf,g \bigr\rangle = \int_{\mathbb{S}^1}
\chi_n(s)h(f,g) (s)\,d\sigma_U(s),
\]
where $h\dvtx H\times H \to
L^1(\sigma_U)$ is a sesquilinear map;
\item[(b)]$\sigma_U$ is minimal with that property, in the sense that
it satisfies $\sigma_U \ll\sigma$ for any measure $\sigma$ on $\mathbb{S}^1$
satisfying (a).
\end{longlist}

In (b) above and throughout the paper, we write $\mu_1 \ll\mu_2$ to
indicate that the measure $\mu_1$ is absolutely continuous with respect
to $\mu_2$. If $\mu_1 \ll\mu_2$ and $\mu_2 \ll\mu_1$, we say they are
in the same measure class.

The spectral type $\sigma_U$ is defined only up to measure class.
Existence of $\sigma_U$ is a formulation of \emph{the scalar spectral
theorem}. 

For a measure-preserving transformation $T$, The \emph{spectral
type} of $T$ $\sigma_T$ is the spectral type of the associated
unitary operator $U_T$ on $L^2(\mu)$. For a probability preserving
transformation $S$, the \emph{restricted spectral type} is the
spectral type the unitary operator $U_S$ restricted to
$L^2$-functions with integral zero.

Our brief exposition here follows Section 2.5 of~\cite{aarobook}.

\subsection{Poisson processes and the Poisson suspension}\label{subsecpoissonprocesses}
For a standard $\sigma$-finite measure space $(X,\mathcal{B},\mu)$,
$(X^*,\mathcal{B}^*,\mu^*)$ denotes the
associated \emph{Poisson point process}, which we now describe. $X^*$
is the space of countable subsets of $X$. We will typically denote an
element of $X^*$ by $\omega$, $\omega_1$, $\omega_2$ and so on. The
$\sigma$-algebra $\mathcal{B}^*$ is generated by sets of the form 
%
\begin{equation}
\label{eqgensigmaalgebraBstar} \bigl[ |\omega\cap B\rrvert =n\bigr]:= \bigl\{\omega\in
X^* \dvtx |\omega\cap B| = n\bigr\}
\end{equation}
for $n \ge0$ and $B \in\mathcal{B}$.


The probability measure $\mu^*$ is
is uniquely defined by requiring that for any
pairwise disjoint $A_1,A_2,\ldots,A_n \in\mathcal{B}$,
if $\omega\in X^*$ is sampled according to $\mu^*$, then
$|\omega\cap A_i|$ are jointly independent random variables\vadjust{\goodbreak} 
individually distributed Poisson with expectation $\mu(A_i)$
%
\begin{equation}
\label{eqpoissondef} \mu^* \bigl(|\omega\cap A|=k \bigr)=e^{-\mu(A)}\frac{\mu(A)^k}{k!}.
\end{equation}

The underlaying measure $\mu^*$ is called the \emph{intensity} of the
Poisson process. We will assume that the measure $\mu$ has no atoms,
namely $\mu(\{x\})=0$ for any $x \in X$. This is a necessary and
sufficient condition to avoid multiplicity of points almost surely with
respect to $\mu^*$.

A Poisson point process can be defined on very general measure spaces,
under milder assumptions than ``standard.''
Details of the construction and general properties of Poisson processes
can be found, for instance, in \cite
{kingmanpoissonbook,kingmanpoissonprocessrevisted}.

To make various measurability statements in the following sections more
transparent, we assume the following technical condition:
There is a fixed sequence $\{\beta_n\}_{n=1}^\infty$ of countable
partitions of $X$ into $\mathcal{B}$-measurable sets,
such that $\beta_{n+1}$ refines $\beta_n$, with the additional
property that the mesh of these partitions goes to
$0$, namely,
\[
\lambda(\beta_n):=\sup\bigl\{ \mu(B) \dvtx B \in\beta_n
\bigr\} \to0 \qquad\mbox{as } n \to\infty.
\]
We assume that $\mathcal{B} = \bigvee_{n=1}^\infty\sigma(\beta_n)$ is
the $\sigma$-algebra generated by the union of these partitions. For
instance, if $(X,\mathcal{B},\mu)$ is the real line with Lebesgue
measure on the Borel sets, we can take $\beta_n$ to be the partition
into half-open intervals with endpoints on the lattice $\frac
{1}{2^n}\mathbb{Z}$.

The $\sigma$-algebra $\mathcal{B}^*$ can now be defined by
\[
\mathcal{B}^* = \bigvee_{n=1}^\infty
\beta_n^*,
\]
where $\beta_n^*$ is the $\sigma$-algebra generated by sets of the form
\eqref{eqgensigmaalgebraBstar} with $B \in\beta_n$ and $n \in\{
0,1,2,\ldots\}$. Different sequences $\beta_n$ with the above
properties will not change the completion with respect to $\mu^*$ of
the resulting $\sigma$-algebra $\mathcal{B}^*$.


The \emph{Poisson suspension} of a measure preserving map $T\dvtx X \to
X$, is the natural map obtained by applying $T$ on $X^*$. As in \cite
{roy2009}, we denote it by $T_*\dvtx X^* \to X^*$. This transformation is
formally defined by
\[
T_*(\omega)=\bigl\{T(x) \dvtx x \in\omega\bigr\}.
\]
$T_*$ is a probability-preserving transformation of $(X^*,\mathcal
{B}^*,\mu^*)$.


The following proposition relates the spectral measures of $T$ and
$T_*$~\cite{roy2009}:
%
\begin{prop}
\label{proppoissonspectral}
If $\sigma$ is the spectral-type of $T$,
the restricted spectral type of $T_*$ is given by
\[
\sigma_{T_*} = \sum_{n \ge1} \frac{1}{n!}
\sigma^{\otimes n}.
\]
\end{prop}

It is a classical result that a probability-preserving transformation
is ergodic if and only if
its restricted spectral type has no atom at $\lambda=1$, and is\vadjust{\goodbreak}
\emph{weakly mixing} if and only if its restricted spectral type has no
atoms in~$\mathbb{S}^1$ (this property is also equivalent to ergodicity of
$T\times T$).
It follows that $T_*$ is ergodic
if and only if $T_*$ is weakly mixing if and only if there are no
$T$-invariant sets of
finite measure in $\mathcal{B}^+$~\cite{roy2009}.

In the following sections we will use the map $\pi\dvtx X \times X^* \to
X^*$ given by
%
\begin{equation}
\label{eqpifactordef} \pi(x,\omega) = \{x\} \cup\omega.
\end{equation}

The map $\pi$ defined by \eqref{eqpifactordef}
is a measurable map from between the measure spaces $(X\times X^*,
\mathcal{B} \otimes\mathcal{B}^*)$ and
$(X^*,\mathcal{B}^*)$.
This is can be verified directly using the following equalities of sets:
\[
\pi^{-1} \bigl[ |\omega\cap A| = 0 \bigr] = (X\setminus A) \times\bigl[ |\omega\cap
A| = 0\bigr]
\]
and
\[
\pi^{-1} \bigl[ |\omega\cap A | = n \bigr] = \bigl((X \setminus A) \times\bigl[ |
\omega\cap A| = n\bigr] \bigr) \cup \bigl(A \times \bigl[ |\omega\cap A| \in\{n-1,n\}
\bigr] \bigr)
\]
for $A \in\mathcal{B}$ and $n \in\mathbb{N}$.

In fact, $\pi$ is a \emph{$\infty$-factor map} between the measure
preserving maps $T \times T_*$ and~$T_*$,
in the sense of Chapter $3$ of~\cite{aarobook}: This means that $\pi
\circ T_* = (T \times T_*) \circ\pi$ and for $A \in\mathcal{B}^*$
\[
\bigl(\mu\times\mu^*\bigr) \circ\pi^{-1} (A)= %
\cases{ 0, &\quad
$\mbox{ if } \mu^*(A)=0$,\vspace*{2pt}
\cr
\infty,& \quad$\mbox{otherwise.}$} %
\]
%
\section{Ergodicity of Poisson product for conservative transformations}
\label{secproofpoissonproductergodic}

We now provide a proof of Theorem~\ref{thmpoissonproductergodic}.
The argument we use is an adaptation of~\cite{aaronadkarni1987}.
To prove our result, we invoke the following condition for
ergodicity of Cartesian products, due to M. Keane:

%
\begin{theorem*}[(The ergodic multiplier theorem)]
Let $S$ be a probability preserving transformation and $T$ a
conservative, ergodic, nonsingular transformation.
$S \times T$ is ergodic if and only if $\sigma_S(e(T))=0$, where:
\begin{itemize}
\item{$\sigma_S$ is the restricted spectral type of $S$;}
\item{$e(T)$ is the group of $L_\infty$-eigenvalues of $T$.}
\end{itemize}
\end{theorem*}

A proof of this result is provided, for instance, in Section $2.7$ of
\cite{aarobook}.

By Proposition~\ref{proppoissonspectral}, the restricted
spectral-type of the Poisson suspension $T_*$
is a linear combination of convolution powers of the spectral type of $T$.

We make use of the following basic lemma about convolution of
measures and equivalence of measure classes. A short proof is provided
here for the
sake of completeness:

\begin{lemma}\label{lemconvolutionrespectsmeasureclass}
Let $\mu_1$ and $\mu_2$ be Borel probability measures on
$\mathbb{S}^1$ with the same
null-sets. For any Borel probability measure $\nu$ on $\mathbb{S}^1$,
the measures $\mu_1 * \nu$ and $\mu_2 * \nu$ have the same
null-sets.
\end{lemma}

\begin{pf}
We will prove that $\mu_1 \ll\mu_2$ implies that $\mu_1 * \nu\ll\mu_2 * \nu$ which suffices by symmetry.

We assume $\mu_1 \ll\mu_2$, and show that for any $\varepsilon>0$, there
exists $\delta>0$ so that any set $A \in\mathcal{P}(\mathbb{S}^1)$
with $(\mu_1 * \nu)(A) \ge\varepsilon$ has $(\mu_2 * \nu)(A) \ge\delta$.

Fix $\varepsilon>0$ and choose any $A \in\mathcal{B}(\mathbb{S}^1)$ with
$(\mu_1 * \nu)(A) \ge\varepsilon$. It follows that
\[
\nu \biggl(\biggl\{ x \in\mathbb{S}^1 \dvtx \mu_1(A
\cdot x) \ge\frac{\varepsilon
}{2} \biggr\} \biggr) \ge\frac{\varepsilon}{2}.
\]

Since $\mu_1 \ll\mu_2$, there exists $\delta'>0$ so that
$\mu_1(B) \ge\frac{\varepsilon}{2}$ implies $\mu_2(B) \ge \delta'$.
Thus,
\[
\nu \bigl(\bigl\{ x \in\mathbb{S}^1 \dvtx \mu_2(A\cdot
x) \ge\delta' \bigr\} \bigr) \ge\frac{\varepsilon}{2}.
\]

It follows that
$(\mu_2 * \nu)(A) \ge\delta'\cdot\frac{\varepsilon}{2}$, which
establishes the claim with $\delta= \delta'\cdot\frac{\varepsilon}{2}$.
\end{pf}


From this we deduce the following lemma.

\begin{lemma}\label{lemeignvaluesactnonsingularly}
Let $T$ be a conservative, 
measure-preserving transformation.
For any $n \ge1$, the group $e(T)$ acts nonsingularly on 
$\sigma_{T}^{\otimes n}$, the $n$th convolution power of the restricted
spectral type of $T$.
\end{lemma}

\begin{pf}
Our claim is that
%
\begin{equation}
\label{eqeignonsingular} \forall t \in e(T) \qquad\sigma_{T}^{\otimes n}
\sim\delta_t * \sigma_{T}^{\otimes n},
\end{equation}
where $\delta_t$ denotes dirac measure at $t$, and $\sim$ denotes
equivalence of measure classes.
For $n=1$, a proof can be found in~\cite{aaronadkarni1987,hann79}.




Equation \eqref{eqeignonsingular} follows
for $n >1$
by induction using Lemma~\ref{lemconvolutionrespectsmeasureclass},
with $t \in e(T)$, $\sigma_T$ and $\delta_t * \sigma_T$
substituting for $\mu_1$ and $\mu_2$, respectively, and $\sigma_T^{\otimes(n-1)}$ substituting for $\nu$.
\end{pf}

\textit{Completing the proof of Theorem
\ref{thmpoissonproductergodic}}.

By the ergodic multiplier theorem above, proving ergodicity of
the Poisson-product amounts to proving $\sigma_{T_*}(e(T))=0$.
Since $\sigma_{T_*}=\sum_{n \ge1}\frac{1}{n!}\sigma_{T}^{\otimes
n}$, it is sufficient to prove that for all $n \ge1$,
%
\begin{equation}
\label{eqeignvaluesspectralnull} \sigma_{T}^{\otimes n}\bigl(e(T)
\bigr)=0.
\end{equation}
A proof that $\sigma_T(e(T))=0$ is provided in~\cite{hann79}; see also
\cite{aaronadkarni1987}. This is the case $n=1$ of equation \eqref
{eqeignvaluesspectralnull}. We also refer to the discussion in
Chapter $9$ of~\cite{nadkarnispectraldsbook}.

For convenience of the reader and in preparation for the discussion in
Section~\ref{secgroupactions}, we briefly recall the arguments
leading to this result:
Suppose the contrary, $\sigma_{T}(e(T)) >0$. Since $e(T)$ acts
nonsingularly on $\sigma_T$, it follow that
$\sigma_{T}\mid_{e(T)}$ is a quasi-invariant measure on
$e(T)$. Thus,
$e(T)$
can be furnished with\vadjust{\goodbreak} a locally-compact second-countable topology,
respecting the Borel structure inherited from~$\mathbb{S}^1$. Haar
measure on $e(T)$ must be is equivalent to $\sigma_{T}\mid_{e(T)}$.
With respect to this topology, we have that $e(T)$ is a locally compact group,
continuously embedded in $\mathbb{S}^{1}$, where the topological
embedding is also a group embedding.
In this situation, it follows as in~\cite{aaronadkarni1987} that
$e(T)$ is either discrete or $e(T)=\mathbb{S}^1$.
The possibility that $e(T)$ is discrete is ruled out since this would
imply $\sigma_{T}$ has atoms, which
means $T$ has $L^2(\mu)$ eigenfunctions. This is impossible since $T$
is an ergodic transformation preserving an infinite measure.
The alternative is that $e(T)=\mathbb{S}^1$. This is impossible since $e(T)$
is weak Dirichlet, thus must be a null set with respect to Haar measure
on $\mathbb{S}^1$~\cite{schmidtspectra1982}.

To prove the equality in \eqref{eqeignvaluesspectralnull} for $n
>1$, note that the convolution power of an atom-free measure is itself
atom-free and that by Lemma~\ref{lemeignvaluesactnonsingularly}
above $e(T)$ also acts nonsingularly on $\sigma_T^{\otimes n}$. The
result now follows using the same arguments outlined above for the case $n=1$.





%
This completes the proof of Theorem~\ref{thmpoissonproductergodic}.

\section{Nonexistence of equivariant thinning} 
\label{secthinnings}

Here is a formalization of the notion of a (deterministic) \emph
{thinning}. This is a
$\mathcal{B}^*$-measurable map $\Psi\dvtx X^* \to X^*$, satisfying
\[
\mu^*\bigl(\bigl[\bigl|\Psi(\omega) \cap B\bigr| \le|\omega\cap B|\bigr]\bigr)=1 \qquad\forall B
\in \mathcal{B}.
\]

This essentially means that $\Psi$ is a measurable map on the space
$X^*$ of countable sets of $X$,
for which almost-surely $\Psi(\omega) \subset\omega$.

A \emph{Poisson thinning} satisfies the extra condition that $\mu^*\circ\Psi^{-1} =
(\theta\mu)^*$ for some $\theta\in(0,1)$.
By $(\theta\mu)^*$ we mean the measure on $(X^*,\mathcal{B}^*)$ which
corresponds to a Poisson process with intensity given by $\theta\cdot
\mu$. In other words, the law of the countable set $\Psi(\omega)$ is
that of a lower-intensity Poisson process.

Given a measure preserving transformation $T\dvtx X \to X$, a thinning
$\Psi$ is called \emph{$T$-equivariant} if $\Psi\circ T_* = T_* \circ
\Psi$.
A thinning $\Psi$ is \emph{trivial} if
\[
\mu^*\bigl( \bigl[\Psi(\omega)= \varnothing\bigr]\bigr)=1 \quad\mbox{or}\quad \mu^*
\bigl( \bigl[\Psi(\omega)= \omega\bigr]\bigr)=1.
\]

\begin{prop}\label{propnopoissonthinning}
Let $T$ be a group-action by measure preserving transformations.
If  $T \times T_*$ is ergodic,
there does not exist a nontrivial $T$-equivariant thinning.
\end{prop}
\begin{pf}
Suppose by contradiction that $\Psi$ is a nontrivial $T$-equivariant thinning.
Consider the set
%
\begin{equation}
A = \bigl\{ (x,\omega) \in X \times X^* \dvtx x \in\Psi\bigl(\omega\cup\{x\}
\bigr)\bigr\}.
\end{equation}
%

Measurability of the set $A$ is verified by the following:
\[
A = \bigcap_{n=1}^\infty\bigcup
_{B \in\beta_n} \bigl(B \times X^* \bigr)\cap \bigl((\Psi\circ
\pi)^{-1}\bigl[|\omega\cap B| >0\bigr] \bigr) \mod\mu \times\mu^*,
\]
where $\{\beta_n\}_{n=1}^\infty$ is a ``decreasing net'' of countable
partitions, as in Section~\ref{secprelim}.\vadjust{\goodbreak}

Since $\Psi$ is $T$-equivariant, the set A is a $T \times T_*$
invariant set. By ergodicity of $T \times T_*$, either $(\mu\times\mu^*)(A)=0$ or $(\mu\times\mu^*)(A^c)=0$.

Intuitively, $A$ is the subset of $X\times X^*$ where applying the
thinning $\Psi$ on the union of the ``indistinguishable points'' with
the ``distinguished point'' does not delete the distinguished point. We
will complete the proof by showing that this
implies that the thinning $\Psi$ is trivial.

For $j \in\mathbb{N}$, define $\pi_{(j)}\dvtx \overbrace{X\times\cdots
\times X}^j \times X^* \to X^*$ by
\[
\pi(x_1,\ldots,x_j,\omega)=\bigcup
_{k=1}^j\{x_k\}\cup\omega.
\]

$\pi_{(j)}$ is $\mathcal{B}^{\otimes j} \otimes\mathcal{B}^*$-measurable.
This follows from measurability of the map $\pi$ given by \eqref
{eqpifactordef}, which coincides with $\pi_{(1)}$.

For any $B \in\mathcal{B}$ with $0<\mu(B)< \infty$, and $j \in\mathbb
{N}$, we consider the following probability measures:
\begin{longlist}[(iii)]
\item[(i)]
\[
\mu^*_{B,j}(\cdot):= \mu^* \bigl(\cdot \mid\bigl[(\omega\cap B) =j
\bigr]\bigr).
\]
This is a probability measure on $(X^*,\mathcal{B}^*)$ corresponding to
a Poisson process with intensity $\mu$, conditioned to have exactly $j$
points in the set $B$,\vspace*{4pt}
\item[(ii)]
\[
\hat{\mu}_{B,j}(\cdot):=\frac{ (\mu\times\mu^*)\mid_{B \times
[(\omega\cap B) =j]} }{\mu(B)\cdot\mu^*([\omega\cap B] =j)}(\cdot).
\]

$\hat{\mu}_{B,j}$ is a probability measure on $X\times X^*$ given by
the product of a random point in $B$, distributed according to $\mu\mid_B$ and an independent Poisson process with intensity $\mu$,
conditioned to have exactly $j$ points inside the set $B$,\vspace*{4pt}
\item[(iii)]
\[
\tilde{\mu}_{B,j}(\cdot):= \frac{\overbrace{\mu\mid_B\times\cdots
\times\mu\mid_B}^j\times(\mu\mid_{B^c})^* }{\mu(B)^j}(\cdot).
\]
This is the probability on $(X^j \times X^*,\mathcal{B}^{\otimes j}
\otimes\mathcal{B}^*)$ which corresponds to $j$ independent random
points identically distributed according to $\mu\mid_B$ and an
independent Poisson process of intensity $\mu\mid_{B^c}$.
\end{longlist}

From the properties of the Poisson process, it directly follows that
the probability measures defined above are related as follows:

\begin{equation}
\label{eqpoissoncondiid} \hat{\mu}_{B,j}\circ\pi^{-1}= \tilde{
\mu}_{B,j+1} \circ\pi_{(j)}^{-1}= \mu^*_{B,j+1}
\end{equation}
and
%
\begin{equation}
\label{eqpoissoncondiid2} \hat{\mu}_{B,j}= \tilde{\mu}_{B,j+1}
\circ\pi_{[2,j]}^{-1},\vadjust{\goodbreak}
\end{equation}
where $\pi_{[2,j]}\dvtx \overbrace{X\times\cdots\times X}^j \times X^* \to
X\times X^*$ is given by
\[
\pi_{[2,j]}(x_1,\ldots,x_j,\omega)=
\Biggl(x_1,\bigcup_{k=2}^j
\{x_k\} \cup\omega \Biggr).
\]


In particular, it follows that $\pi_{(j)}$ is a nonsingular map for all
$j \ge1$, in the sense that
the inverse image of a $\mu^*$-null set is always $\overbrace{\mu
\times\cdots\mu}^j \times\mu^*$-null.





Assuming $\Psi$ is not a trivial thinning implies that
there exist $B \in\mathcal{B}$ with $0<\mu(B)<\infty$ 
so that
\[
\mu^* \bigl( 0 < \bigl| \Psi(\omega) \cap B\bigr| < | \omega\cap B| \bigr)>0.
\]
It follows that for some $j >1$,
%
\begin{equation}
\label{eqprobdelete} \mu^*_{B,j} \biggl( 0 < \frac{| \Psi(\omega) \cap B|}{ | \omega\cap B|} < 1
\biggr)>0.
\end{equation}

Now by \eqref{eqpoissoncondiid} and \eqref{eqpoissoncondiid2},
using symmetry of $\tilde{\mu}_{B,j}$ with respect to the variables
$(x_1,\ldots,x_j)$, it follows that the probability
$\hat{\mu}_{B,j} ( x \in\Psi(\pi(x,\omega)) )$ is equal to
the expectation of $\frac{| \Psi(\omega) \cap B|}{ | \omega\cap B|}$
under $\mu^*_{B,j}$. By \eqref{eqprobdelete}\vspace*{1pt} this expectation must be
strictly positive and smaller than one. This contradicts triviality of
the set~$A$: Either $(\mu\times\mu^*)(A)=0$ in which case $\hat{\mu
}_{B,j} ( x \in\Psi(\pi(x,\omega)) )=0$ or $(\mu\times\mu^*)(A^c)=0$ in which case $\hat{\mu}_{B,j} ( x \in\Psi(\pi(x,\omega
)) )=1$.
\end{pf}

\section{Nonexistence of equivariant allocation and matching}\label
{secallocationmatching}

The aim of this section is to establish 
the nonexistence of $T$-equivariant Poisson allocation and Poisson
matching, under an ergodicity assumption of a certain extension of $T$.
Combined with Theorem~\ref{thmpoissonproductergodic}, this will
establish the last part of Theorem~\ref{thmnopoissonthin}.



We begin with an intermediate result about measure-preserving systems.
Consider
a measurable function $ \Phi\dvtx X \to L^1(\mu)$, sending $x \in X$ to $\Phi_x \in L^1(\mu)$, 
which is $T$-equivariant in the sense that $\Phi_{Tx} \circ T= \Phi_x
$. Such a function $\Phi$ can be interpreted as a $T$-equivariant
``mass allocation'' scheme.
For instance, on $X=\mathbb{R}^d$ with Lebesgue measure, $\Phi_x(y) =
1_{B_1(x)}(y)$ and $\Phi_x(y) = \exp(-\|x-y\|)$ both define
isometry-equivariant ``mass allocations.'' The later can be considered
a ``fractional allocation,'' in the sense that it obtains values in the
interval $(0,1)$. Nonexistence of $T$-equivariant Poisson allocation
and Poisson matching will be a consequence of the following:

\begin{prop}\label{propnomassallocation}
$\!\!\!\!\!$Let $T$ be a measure-preserving group action on $(X,\mathcal{B},\mu)$.
If $T \times T_*$ is ergodic, and $\mu(X)=\infty$, 
any $T$-equivariant
measurable function
$ \Phi\dvtx X \to L^1(\mu)$
must be equal to $0$ $\mu$-a.e.
\end{prop}

\begin{pf}
Suppose $\Phi\dvtx X \to L^1(\mu)$ satisfies $\Phi_{Tx} \circ T= \Phi_x$.
Note that ergodicity of $T$ implies that $\|\Phi_x\|_{L^1(\mu)} $ is
constant $\mu$-a.e, as this is a\vadjust{\goodbreak} $T$-invariant function.
Consider the function $F\dvtx X \times X^* \to\mathbb{R}$ given by
\[
F(x,\omega) = \sum_{y \in\omega} \bigl|\Phi_x(y)\bigr|.
\]
%
We verify that $F$ indeed coincides with a $\mathcal{B}\otimes\mathcal
{B}^*$-measurable function on a set of full $\mu\times\mu^*$-measure.

Indeed,
\[
\Phi_x = \sum_{B \in\beta_1} \sum
_{ y \in\omega\cap B} \bigl| \Phi_x(y)\bigr|,
\]
by Martingale convergence,
\[
\sum_{ y \in\omega\cap B} \bigl| \Phi_x(y)\bigr| =
\lim_{n \to\infty} E_{\mu
^*} \biggl(\sum_{ y \in\omega\cap B}
\bigl| \Phi_x(y)\bigr| \mid\beta_n^* \biggr)
\]
for $\mu\times\mu^*$-almost-every $(x,\omega)$.
For $B \in\beta_1$ and $n \ge1$ we have
\[
E_{\mu^*} \biggl(\sum_{ y \in\omega\cap B} \bigl|
\Phi_x(y)\bigr| \mid\beta_n^* \biggr) = \sum
_{D \in\beta_n \cap B}E_{\mu^*} \biggl(\sum
_{y \in
(\omega\cap D)}\bigl|\Phi_x(y)\bigr| \biggr),
\]
and the right-hand side is clearly $\mathcal{B} \times\beta_n^*$-measurable.







Let
\[
\tilde{F}(x):= \int\bigl|F(x,\omega)\bigr| \,d\mu^*(\omega) = \int\sum
_{y \in
\omega} \bigl|\Phi_x(y)\bigr| \,d\mu^*(\omega),
\]
and it follows from the definition of $\mu^*$ that
$\tilde{F}= \| \Phi_x\|_{L^1(\mu)}$. Thus, by ergodicity of $T$,
$\tilde{F}$
is equal to a nonzero (finite) constant $\mu$-almost everywhere. In
particular, $F$ is finite $\mu\times\mu^*$-almost everywhere.

Observe that $F$ is $T\times T_*$-invariant, so by ergodicity of $T
\times T_*$ must be constant $\mu\times\mu^*$-a.e. On the other hand,
for any $\varepsilon>0$ and $M >0$, we have $F(x,\omega) > M$ whenever
$(x,\omega) \in X \times X^*$ satisfy $|\omega\cap\{y \in X \dvtx   |\Phi_x(y)| > \varepsilon\}| > \frac{M}{\varepsilon}$.
%
From the definition of the Poisson process, it thus follows that
\[
\bigl(\mu\times\mu^*\bigr) \bigl( [ F \ge M] \bigr) \ge\mu\bigl(\bigl\{ x \in X \dvtx
\Vert \Phi_x\Vert_{L^1(\mu)} \ge\varepsilon\bigr\}\bigr) \cdot
\frac{\varepsilon^{{M}/{\varepsilon
}}}{M!}\exp \biggl( -\frac{M}{\varepsilon} \biggr).
\]

Because the right-hand side is strictly positive for any $M >0$,
whenever $\varepsilon>0$ is sufficiently small,
it follows that $F$ is not essentially bounded, which contradicts~$F$
being almost-everywhere constant.
\end{pf}

Together with Theorem~\ref{thmpoissonproductergodic}, Proposition
\ref{propnomassallocation}, immediately gives the following
corollary, which does not seem to involve Poisson processes at all:

\begin{corollary}
Let $T\dvtx X\to X$ be a conservative and ergodic measure preserving
transformation of $(X,\mathcal{B},\mu)$ with $\mu(X)=\infty$.
Any measurable function
$ \Phi\dvtx X \to L^1(\mu)$
satisfying $\Phi_{Tx}\circ T = \Phi_x $ must be equal to $0$ $\mu$-a.e.\vadjust{\goodbreak}
\end{corollary}



We now turn to define and establish a nonexistence result for
equivariant Poisson allocations:

By a \emph{Poisson allocation rule} we mean a $\mathcal{B}^* \otimes
\mathcal{B}$-measurable map $\Upsilon\dvtx X \times X^* \to L^1(\mu)$
satisfying the following properties:
\begin{longlist}[(A1)]
\item[(A1)]{\emph{nonnegativity:} $\Upsilon_{(x,\omega)}(y) \ge 0$;}
\item[(A2)]{\emph{partition of unity:} $ \sum_{x \in\omega}(y) \Upsilon_{(x,\omega)} = 1$ $\mu^*$-a.e.;}
\item[(A3)]{$\Upsilon_{(x,\omega)} \equiv0$ if $x \notin\omega$.}
\end{longlist}
%

If $x \in\omega$, we think of $\Upsilon_{(x,\omega)}$ as the ``the
cell allocated to $x$.'' Properties (A1) and (A2) above guarantee
that $\Upsilon$ essentially takes values in the interval $[0,1]$.
The three above properties together express the statement that
$\Upsilon_{(\cdot,\omega)}$ corresponds to a partition of $X$ up to a
null set between the points in $\omega$, which assigns each $x \in
\omega$ finite mass. For a ``proper'' allocation, we would require that
$\Phi_{(x,\omega)}$ only takes values in $\{0,1\}$, but this extra
requirement is not necessary in order to prove our result.

For it is often useful to consider a wider class of Poisson allocation
rules, where $\Upsilon_{(x,\omega)}$ is undefined for a null set of
$(x,\omega)$'s, and $\Upsilon$ is only measurable with respect to the
$\mu\times\mu^*$-completion of the $\sigma$-algebra $\mathcal{B}^*
\otimes\mathcal{B}$. However, conditions~(A2) and~(A3) above apply
to $\mu\times\mu^*$-null sets, so we need to be careful and restate
them as follows:
\begin{longlist}[(A3$'$)]
\item[(A1)]{\emph{nonnegativity:} $\Upsilon_{(x,\omega)}(y) \ge 0$;}
\item[(A2$'$)]{\emph{partition of unity:} $ \int_X \Upsilon_{(x,\omega)}
\,d\mu(x) = 1$ $\mu^*$-a.e.;}
\item[(A3$'$)]{$\int_A \Upsilon_{(x,\omega)}\,d\mu(x) \equiv0$ $\mu^*$-a.e
on $\{ \omega\in X^* \dvtx   \omega\cap A = \varnothing\}$ whenever $A \in
\mathcal{B}$.}
\end{longlist}
%

A poisson allocation $\Upsilon$ is $T$-equivariant if $\Upsilon_{(Tx,T_*\omega)}\circ T = \Upsilon_{(x,\omega)} $.







\begin{prop}
\label{propnopoissonallocation}
Let $T$ be a group-action by measure preserving transformations, and denote
$S:=T \times T_*$.
If $S \times S_*$ is ergodic,
there does not exist a $T$-equivariant Poisson-allocation.
\end{prop}

\begin{pf}
Given a Poisson allocation $\Upsilon\dvtx X \times X^* \to L^1(\mu)$, we
will define a $T \times T_*$-equivariant
function $\Phi\dvtx X\times X^* \to L^1(\mu\times\mu^*),$ 
which by ergodicity of $S=T \times T_*$ will contradict Proposition \ref
{propnomassallocation}.
This is given by
\[
\Phi_{(x,\omega)}(y,\omega_2)= \Upsilon_{(x,\omega\cup\{x\})}(y).
\]


It follows directly that
\[
\| \Phi_{(x,\omega)}\|_{L^1(\mu\times\mu^*)} = \| \Upsilon_{(x,\omega
\cup\{x\})}
\|_{L^1(\mu)},
\]
%
which is positive and finite $\mu\times\mu^*$-a.e.

Measurability of $\Phi$ follows from the measurability assumptions on
$\Upsilon$ and from measurability of the map $(x,\omega) \to\{x\} \cup
\omega$.
\end{pf}

We now consider the existence of equivariant Poisson matching
schemes:

Given a pair of independent Poisson processes realizations a
(deterministic) \emph{Poisson matching}
assigns a perfect matching (or bijection) between the points of the two
realizations, almost surely.
To formalize this we define a Poisson matching as a
measurable-function $\Psi\dvtx X^* \times X^* \to(X \times X)^*$, 
satisfying the following:

\begin{longlist}[(M1)]
\item[(M1)]
\[
\mu^* \bigl( \bigl\{ \omega_2\in X^* \dvtx \bigl|\Psi(
\omega_1,\omega_2) \cap (B_1 \times
B_2) \bigr| \le\min\bigl\{|\omega_1 \cap B_1 |,|
\omega_2 \cap B_2 |\bigr\} \bigr\} \bigr) =1
\]
for $\mu^*$-a.e $\omega_1$ and all $B_1,B_2 \in\mathcal{B}$;\vspace*{2pt}
\item[(M2)]
\[
\mu^* \bigl( \bigl\{\omega_2 \in X^* \dvtx\bigl |\Psi(
\omega_1,\omega_2) \cap (B_1 \times X)\bigr | = |
\omega_1 \cap B_1 | \bigr\} \bigr) =1
\]
for $\mu^*$-a.e $\omega_1$ and all $B_1 \in\mathcal{B}$;\vspace*{2pt}
\item[(M3)]
\[
\mu^* \bigl( \bigl\{\omega_1 \in X^* \dvtx \bigl|\Psi(
\omega_1,\omega_2) \cap(X \times B_2)\bigr | = |
\omega_2 \cap B_2 | \bigr\} \bigr) =1
\]
for $\mu^*$-a.e $\omega_2$ and all $B_2 \in\mathcal{B}$.
\end{longlist}

\begin{prop}\label{propnopoissonmatching}
Under the assumptions of Proposition~\ref{propnopoissonallocation},
there does not exist a nontrivial $T$-equivariant Poisson matching.
\end{prop}

\begin{pf}
Suppose $\Psi$ is a $T$-equivariant Poisson matching.
We will define a ``fractional'' $T$-equivariant Poisson allocation
$\Upsilon\dvtx X \times X^* \to L^1(\mu)$, contradicting Proposition \ref
{propnopoissonallocation}.

The (implicit) definition of $\Upsilon$ is given by
\begin{equation}
\label{eqmatchingfromallocation} \int_A
\Upsilon_{(x,\omega_1)}(y)\,d\mu(y) = \mu^* \bigl(\bigl\{ \omega_2
\dvtx \bigl|\Psi(\omega_1,\omega_2)\cap\bigl(\{x\} \times A
\bigr)\bigr|>0 \bigr\} \bigr)
\end{equation}
for all $A \in\mathcal{B}$, $\omega_1 \in X^*$ and $x \in X$.

In other words, if $x\in\omega_1$, $\Upsilon_{(x,\omega_1)}$ is the
density with respect to Lebesgue measure of the
conditional distribution of the \emph{partner} of $x$ under the
matching $\Psi$, given~$\omega_1$.
This defines $\Upsilon$ up to a null set.

It follows from the properties of $\Psi$ that $\Upsilon$ satisfies the
conditions (A1), (A2$'$) and (A3$'$) above.

Thus, $\Upsilon$ is indeed a Poisson allocation.
Because $\Psi$ is a $T$-equivariant matching, it follows directly that
$\Upsilon$ is a $T$-equivariant allocation.
\end{pf}

To complete the proof of the last part of Theorem \ref
{thmnopoissonthin}, we note that if $T$ is a conservative and
ergodic measure-preserving transformation, $S=T \times T_*$ is also
conservative and ergodic by Theorem~\ref{thmpoissonproductergodic},
and so $S \times S_*$ is also ergodic, again by Theorem \ref
{thmpoissonproductergodic}.

\section{The leftmost position transformation}
\label{secFROL}
In this section $X=\mathbb{R}_+$
is the set of positive real numbers, $\mathcal{B}$ is the Borel $\sigma
$-algebra on $X$ and $\mu$
is Lebesgue measure on the positive real numbers. $T\dvtx X \to X$ is an
arbitrary conservative, ergodic, Lebesgue-measure-preserving
map of the positive real numbers.\vadjust{\goodbreak}

In order to have a concrete example for such transformation $T$ in
hand, the reader can consider the unsigned version of Boole's
transformation, given by $T(x)=|x-\frac{1}{x}|$.
We define the following function:
%
\begin{equation}
\label{eqt1} t_1\dvtx X^* \to X\qquad \mbox{by } t_1(\omega)=
\inf\omega.
\end{equation}
The map $t_1$ is well defined on a set of full $\mu^*$-measure, namely
whenever $\omega\ne\varnothing$.
Note that
$t_1(\omega)$ is the leftmost point of $\omega$ whenever $\omega$ is a
discrete countable subset of $\mathbb{R}_+$.
The map $t_1$ is $\mathcal{B}^*$-measurable since
\[
t_1^{-1}(a,b) = \bigl\{\omega\in X^* \dvtx \omega
\cap(0,a]=\varnothing\mbox{ and } \omega\cap(a,b) \ne\varnothing\bigr\}.
\]
From this, it also follows directly that
\[
\mu^*\circ t_1^{-1}(a,b) = e^{-\mu(0,a)}
\bigl(1-e^{-\mu(a,b)} \bigr) = e^{-a}-e^{-b}.
\]

In particular it follows that $ \mu^*\circ t^{-1} \ll\mu$.

Define the \emph{leftmost return time} $\kappa\dvtx X^* \to\mathbb{N}\cup\{
+\infty\}$ by
%
\begin{equation}
\label{eqkappa} \kappa(\omega) = \inf\bigl\{k \ge1 \dvtx t_1
\bigl(T_*^k(\omega)\bigr)=T^k\bigl(t_1(\omega
)\bigr)\bigr\}.
\end{equation}
$\mu^*$-almost surely, $\kappa(\omega)$ is the smallest positive
number of iterations of $T_*$ which must be applied to $\omega$ in
order for the leftmost point to return to the leftmost location. A
priori, $\kappa_T$ is could be infinite. Nevertheless, we will soon
show that when~$T$ is conservative and measure preserving,
$\kappa$ is finite $\mu^*$-almost surely. Finally, the
\emph{leftmost position transformation} associated with $T$,
$T_*^\kappa\dvtx \omega\to\omega$, is defined by
\[
T_*^\kappa(\omega)\dvtx =T_*^{\kappa(\omega)}(\omega).
\]

This is the map of
$X_*$ obtained by reapplying $T_*$ till once again there are no
points to the left of the point which was originally leftmost.

The reminder of this section relates the leftmost transformation
associated with~$T$ with the Poisson-product $T \times T_*$.

Let
%
\begin{equation}
\label{eqX0} X_0 = \bigl\{(x,\omega) \in X\times X^* \dvtx \omega
\cap(0,x] = \varnothing\bigr\}.
\end{equation}

The set $X_0$ is simply the subset of $X \times X^*$ in which the
``distinguished point'' is strictly to the left of any
``undistinguished point.''
The formula below verifies 
measurability of $X_0$:
\[
X_0 = \bigcap_{n \in\mathbb{N}} \bigcup
_{q \in\mathbb{Q}} \biggl(\biggl(q-\frac{1}{n},q+\frac{1}{n}
\biggr)\times \biggl\{\omega\in X^* \dvtx \omega\cap \biggl(0,q+\frac{2}{n}
\biggr) = \varnothing\biggr\} \biggr) \mod\mu\times\mu^*.
\]

\begin{prop}\label{propleftmostisinducedproduct}
Let $T\dvtx \mathbb{R}_+\to\mathbb{R}_+$ be conservative and
Lebesgue-meas\-ure-preserving. Then the leftmost position
transformation associated with $T$ is well defined\vadjust{\goodbreak} and is isomorphic
to the induced map of the Poisson product on the set $X_0$ defined by
equation \eqref{eqX0},
\[
\bigl(X^*,\mathcal{B}^*,\mu_*,T_*^\kappa\bigr) \cong
\bigl(X_0,\mathcal{B}_0, \mu_0,(T\times
T_*)_{X_0} \bigr),
\]
where $\mu_0=(\mu\times\mu_*)\mid_{X_0}$ is the restriction of the
measure product $\mu\times\mu_*$ to the set $X_0$, and
$\mathcal{B}_0 =  (\mathcal{B}\otimes\mathcal{B}^* )\cap
X_0$ is the restriction of the $\sigma$-algebra on the product space
to subset of $X_0$.

In particular, 
$\mu_0(X_0)=1$,
so $(X_0,\mathcal{B}_0,\mu_0)$ is a probability space.
\end{prop}

\begin{pf}
Consider the map $\pi_0\dvtx X_0 \to X^*$ which is the restriction to $X_0$
of the map $\pi(x,\omega) =
\{x\}\cup\omega$ described in Section~\ref{subsecpoissonprocesses} above.

For a nonempty, discrete $\omega\in X^*$ we have
\[
\pi_0^{-1}(\omega)=\bigl(t_1(\omega),\omega
\setminus t_1(\omega)\bigr).
\]
Thus $\pi_0$ is invertible on a set of full $\mu^*$-measure in $X^*$.

As $T$ is conservative and $T_*$ is a probability preserving
transformation, the Poisson product $T\times T_*$ is also
conservative. We will show below that $\mu\times\mu^*(X_0)>0$.
Therefore, the return time $\varphi_{X_0}$ is finite almost
everywhere on $X_0$.

Since $\kappa\circ\pi_0 = \pi_0 \circ\varphi_{X_0}$, it follows that
$\kappa$ is finite $\mu^*$-a.e.

We also have
\[
\pi_0\bigl(T^nx,T_*^n\omega
\bigr)=T_*^n\bigl(\pi_0(x,\omega)\bigr)
\]
whenever $(x,\omega)$ and $(T^nx,T_*^n\omega)$ are in $X_0$.
Thus,
\[
\pi_0 \circ(T \times T_*)_{X_0} = T_*^\kappa\circ
\pi_0.
\]

It remains to check that $\pi_0^{-1}\mu^*=\mu_0$.
It is sufficient to
verify that $\mu^*(A)=\mu_0(\pi_0^{-1}(A))$ for sets $A \in
\mathcal{B}^*$ of the form

%
\[
A= \bigcap_{k=1}^N\bigl[ |\omega\cap
A_k | = n_k\bigr],
\]
where $A_i=(a_{i-1},a_{i}]$,
$0=a_0 <a_1 < a_2 <\cdots<a_N$ and $n_k \ge0$ for $k=1,\ldots N$.

Given the definition of $\mu^*$, this amounts to an exercise in
elementary calculus.
By
definition of $\mu^*$,
\[
\mu^*(A)=\prod_{k=1}^N \frac{\mu(A_k)^{n_k}}{n_k!}
\exp \bigl(-\mu(A_k) \bigr),
\]
which simplifies to
%
\begin{equation}
\label{eqmuA} \mu^*(A) =\exp(-a_N)\prod
_{k=1}^{N}\frac{(a_{k}-a_{k-1})^{n_k}}{n_k!}.
\end{equation}

Assuming the $n_k$'s are not all zero, let $k$ the smallest index for
which $n_k > 0$. We have
\begin{eqnarray*}
\pi_0^{-1}(A) &= &\bigcap_{j \ne k}
\bigl(X \times\bigl[|\omega\cap A_j| = n_j\bigr] \bigr)\\
&&{}\cap\bigcup
_{ x \in A_{k}}\{x\}\times \bigl(\bigl[\bigl|\omega\cap[a_{k-1},x
)\bigr|=0\bigr] \cap\bigl [\bigl|\omega\cap[x,a_{k})\bigr|=n_k-1\bigr] \bigr).
\end{eqnarray*}

Thus
\[
\mu_0\bigl(\Phi^{-1}(A)\bigr)=T_0 \int
_{A_k}\exp\bigl(-(x-a_{k-1})\bigr)\exp
\bigl(-(a_{k}-x)\bigr)\frac
{(a_{k}-x)^{n_{k}-1}}{(n_{k}-1)!}\,dx,
\]
where
\[
T_0=\prod_{j\ne
k}\frac{(a_{j}-a_{j-1})^{n_j}}{n_j!}\exp
(a_{j}-a_{j-1} ).
\]

Integrating this rational function of a single variable, we see that
the last expression is equal to the expression
on right-hand side of \eqref{eqmuA}.

In particular, it follows that $\mu_0(X_0)=1$.

%
It remains to check the case that $n_k=0$ for all $k=1,\ldots, N$: In
this case then $A= [ \omega\cap(0,a_N]=0 ]$ and
\[
\pi_0^{-1}(A)= \bigl\{(x,\omega) \in X_0
\dvtx x > a_n \bigr\}.
\]
Thus
\[
\mu_0\bigl(\pi_0^{-1}(A)\bigr)=\int
_{[a_N,\infty)} e^{-\mu[x,\infty)}\,d\mu(x)=\exp(-a_N),
\]
which is equal to
$\mu^*(A)$.
\end{pf}

\begin{corollary}\label{corleftmostergodic}
$\!\!\!$Let $T\dvtx \mathbb{R}_+\to\mathbb{R}_+$ be a conservative and ergodic
Lebesgue-measure-preserving transformation. Then the leftmost position
transformation $T_*^\kappa\dvtx (\mathbb{R}_+)^* \to(\mathbb{R}_+)$ is an
ergodic probability preserving transformation.
\end{corollary}

\begin{pf}
Let $T$ be as above.
By Proposition~\ref{propleftmostisinducedproduct}, $T_*^\kappa$ is
isomorphic to the map obtained by inducing the Poisson product $T
\times T_*$ onto the set $X_0$. It is well known that inducing a
conservative and ergodic transformation on a set of positive measure
results in an ergodic transformation. By Theorem~\ref
{thmpoissonproductergodic}, $T\times T_*$ is indeed ergodic.
\end{pf}

It would be interesting to establish other ergodic properties of
$T^\kappa$. For example, what conditions on $T$ are required for
$T^\kappa_*$ to be weakly mixing?

\section{Poisson-products and measure-preserving group actions}\label
{secgroupactions}
The purpose of this section is to discuss counterparts of our pervious
results on ergodicity of Poisson products,
and various equivariant operations in the context of a group\vadjust{\goodbreak}
of measure preserving transformations.
Some motivating examples for this are groups of
$\mathbb{R}^n$-isometries, which naturally act on $\mathbb{R}^n$
preserving Lebesgue measure.

Briefly recall the basic setup:
We fix a topological group $\mathbb{G}$ and a $\sigma$-finite measure
space $(X,\mathcal{B},\mu)$. A
measure-preserving $\mathbb{G}$-action $T$ on the $\sigma$-finite
measure space $(X,\mathcal{B},\mu)$
is a representation $g \mapsto T_g \in\operatorname{Aut}(X,\mathcal{B},\mu)$
of $\mathbb{G}$
into the measure preserving automorphisms of $(X,\mathcal{B},\mu)$.

A $\mathbb{G}$-action $T$ is \emph{ergodic} if for some $A \in\mathcal
{B}$, $\mu(T_g A \setminus A)=0$
for all $g \in\mathbb{G}$ then either $\mu(A)=0$ or $\mu(X \setminus A)=0$.

Any measure preserving $\mathbb{G}$-action $T$ induces an action $T_*$
on the Poisson process by probability preserving transformations \cite
{roypoissonpinsker}. The Poisson-product $\mathbb{G}$-action $T
\times T_*$ is thus defined the same way as in the case of a single
transformation.

The proofs of Propositions~\ref{propnopoissonthinning}, \ref
{propnomassallocation},~\ref{propnopoissonallocation} and \ref
{propnopoissonmatching} above are still valid in this generality.

Let us recall the definition of a conservative $\mathbb{G}$-action:
Say $W \in\mathcal{B}$ is a \emph{wandering set} with resect to the
action $T$ of a locally-compact group $\mathbb{G}$ if
$\mu(T(g,W)\cap W)=0$ for all $g$ in the complement of some compact
$K \subset\mathbb{G}$. Call a $\mathbb{G}$-action
\emph{conservative} if there are no nontrivial wandering sets.

If in the statement of Theorem~\ref{thmpoissonproductergodic} we
let $T$ be a conservative ergodic $\mathbb{G}$-action for a group
other than $\mathbb{Z}$, ergodicity of $T\times T_*$ may fail. This
can happen even for conservative and ergodic $\mathbb{Z}^2$-actions,
as we demonstrate in the example below:

Let $a,b \in\mathbb{R}\setminus\{0\}$ with $\frac{a}{b} \notin\mathbb
{Q}$.
Define a $\mathbb{Z}^2$-action $T$ on $\mathbb{R}$ by
\[
T_{(m,n)}(x)=x+am+bn \qquad\mbox{for } (m,n) \in\mathbb{Z}^2.
\]

It is a simple exercise to show that the $\mathbb{Z}^2$-action above
is both conservative and ergodic. Nevertheless, it is easy to see
that $T \times T_*$ is not ergodic, for instance, by noting that
\[
\bigl\{ (x,\omega) \in\mathbb{R}\times\mathbb{R}^* \dvtx (x+1,x-1)\cap \omega=
\varnothing\bigr\}
\]
is a nontrivial $T\times T_*$-invariant set.
Since this action $T$ consists of translations, as noted in the
\hyperref[sec1]{Introduction}, there do exist $T$-equivariant Poisson allocations,
Poisson matchings and Poisson thinning.

Although the example above demonstrates Theorem \ref
{thmpoissonproductergodic} does not generalize, for abelian group
actions most components of the proof given in Section~\ref
{secproofpoissonproductergodic} remain intact. Our next goal is to
explain this, and point out where the proof of Theorem~\ref
{thmpoissonproductergodic} breaks down for the example above:

Let $\mathbb{G}$ be a locally compact \emph{abelian} group, and let
$\widehat{\mathbb{G}}$ denote its dual.
Generalizing the discussion in Section~\ref{secprelim}, the $L^{\infty}$-\emph{spectra}
of a $\mathbb{G}$-action $T$, denoted $\operatorname{Sp}(T)$,
is the set of homomorphisms $\chi\dvtx \mathbb{G} \to\mathbb{C}^*$
such that $f(T_gx)=\chi(g)f(x)$ for
some nonzero $f \in L^{\infty}(X,\mu)$.
In case $\mathbb{G}=Z$, the spectra is simply the group $L^{\infty
}$-eigenvalues. As in the case
$\mathbb{G}=\mathbb{Z}$ discussed earlier, the $L^\infty$-spectra is a
weak-Dirichlet set in
$\widehat{\mathbb{G}}$~\cite{schmidtspectra1982}.

The $L^{2}$-\emph{spectral type} of $T$ is an equivalence class of
Borel measures $\sigma_T$
on $\widehat{\mathbb{G}}$ for any nonzero $f \in L^2(\mu)$ $\sigma_f
\ll\sigma_T$, where the measure $\sigma_f$ is given by
%
\[
\hat\sigma_f(g) = \int f\bigl(T_g(x)\bigr)
\overline{f(x)}\,d\mu(x).
\]
The spectral type of $\sigma_T$ is the minimal equivalence class of
measures on $\widehat{\mathbb{G}}$ with respect to which all the $\sigma_f$'s are absolutely continuous.

With these definitions, Keane's ergodic multiplier theorem above
generalizes as follows:
The product of an ergodic measure preserving $\mathbb{G}$-action $T$
and a probability preserving $\mathbb{G}$-action $S$
is ergodic if and only if $\operatorname{Sp}(T)$ is null with respect to the
restricted spectral type of $\sigma_T$. The discussion in the end of
Section~\ref{secproofpoissonproductergodic} following \cite
{aaronadkarni1987,schmidtspectra1982} still shows that in this case
$\operatorname{Sp}(T)$ must be a locally compact group continuously which embeds
continuously in $\widehat{\mathbb{G}}$. However, when $\mathbb{G} \ne
\mathbb{Z}$, this does not imply that $\operatorname{Sp}(T)$ is either discrete or
equal to $\widehat{\mathbb{G}}$.

Getting back to the example of the $\mathbb{Z}^2$-action $T$ above, we
note that for any $\tau\in
\mathbb{R}$, the function $f_\tau\in L^{\infty}(\mathbb{R})$
defined by
\[
f_\tau(x) = \exp(i \tau x),
\]
is an $L^\infty$ eigenfunction of $T$, since it satisfies
\[
f_\tau\bigl(T_{(m,n)}(x)\bigr) = \exp\bigl(i \tau(x+am+bn)
\bigr)= \chi_{(ta,tb)}(m,n)\exp (i\tau x),
\]
where $\chi_{(a,b)}(m,n)=\exp(i am+ bn)$. The map $t \to
\chi_{(ta,tb)}$ is a continuous group embedding of $\mathbb{R}$ in
$\operatorname{Sp}(T) \subsetneq\widehat{\mathbb{Z}^2}$.

\section*{Acknowledgments}
Thanks to Emmanual Roy for inspiring
conversations and in particular for suggesting the ``leftmost position
transformation'' and asking about its
ergodicity. This work is indebted to Jon Aaronson for numerous
contributions, in particular for recalling the paper
\cite{aaronadkarni1987}, which contains key points of the main
result. To Omer Angel and Ori Gurel-Gurevich, thanks for helpful discussions
about equivariant operations on Poisson processes.

%

%

\printaddresses

\end{document}